\documentclass[12pt]{article}

\usepackage[all]{xy}
\usepackage{amsmath, amssymb,enumitem, amsthm}

\input diagxy

\usepackage[colorlinks=true]{hyperref}
\hypersetup{allcolors=[rgb]{0.1,0.1,0.4}}

\author{Michael Anton Hoefnagel}

\title {Majority categories}
\date{}

\newcount\colveccount
\newcommand*\colvec[1]{
	\global\colveccount#1
	\begin{pmatrix}
		\colvecnext
	}
	\def\colvecnext#1{
		#1
		\global\advance\colveccount-1
		\ifnum\colveccount>0
		\\
		\expandafter\colvecnext
		\else
	\end{pmatrix}
	\fi
}

\newcommand{\Top}{\mathbf{Top}}
\newcommand{\Met}{\mathbf{Met}}

\newcommand{\Lat}{\mathbf{Lat}}
\newcommand{\Rel}{\mathbf{Rel}}

\newcommand{\C}{\ensuremath{\mathbb{C}}}
\newcommand{\D}{\ensuremath{\mathbb{D}}}

\newcommand{\op}{\mathrm{op}}
\newcommand{\Pt}{\mathrm{Pt}}
\newcommand{\Mal}{\mathrm{Mal}}

\newcommand{\Maj}{\mathrm{Maj}}
\newcommand{\Set}{\mathbf{Set}}

\newtheorem{definition}{Definition}
\newtheorem{theorem}{Theorem}
\newtheorem{example}{Example}
\newtheorem{remark}{Remark}
\newtheorem{lemma}{Lemma}
\newtheorem{proposition}{Proposition}
\newtheorem{corollary}{Corollary}

\begin{document}
	\maketitle
	\begin{abstract}
		We introduce the notion of a majority category --- the categorical counterpart of varieties of universal algebras admitting a majority term. This notion can be thought to capture properties of the category of lattices, in a way that parallels how Mal'tsev categories capture properties of the category of groups. Among algebraic majority categories are the categories of lattices, Boolean algebras and Heyting algebras. Many geometric categories such as the category of topological spaces, metric spaces, ordered sets, any topos, ect., are comajority categories (i.e.~their duals are majority categories), and we show that, under mild assumptions, the only categories which are both majority and comajority, are the preorders. Mal'tsev majority categories provide an alternative generalization of arithmetical categories to protoarithmetical categories in the sense of Bourn. We show that every Mal'tsev majority category is protoarithmetical, provide a counter-example for the converse implication, and show that in the Barr-exact context, the converse implication also holds. We can then conclude that a category is arithmetical if and only if it is a Barr-exact Mal'tsev majority category, recovering in the varietal context a well known result of Pixley.
	\end{abstract}
	
	
	\section{Introduction}\label{sec-Introduction}
	A \textit{majority term} in universal algebra is a ternary term $p$, satisfying the equations:
	\begin{align*}
	p(x,x,y) = x, \\
	p(x,y,x) = x, \tag{$*$}\\
	p(y,x,x) = x.
	\end{align*}
	Such a term naturally arises from the theory of congruence distributive varieties: a congruence permutable variety admits a majority term if and only if it is congruence distributive (this result
	was proved by A.~F.~Pixley, see Theorem~2 in \cite{pixley}). In the variety of lattices, the term 
	\[
	p(x,y,z) = (x \wedge y) \vee (x \wedge z) \vee (y \wedge z)
	\]
	is a majority term. If $R$ is a ring satisfying the identity $x^n = x$ for some $n \geqslant 2$ (a finite field for example), then the term 
	\[
	p(x,y,z) = x - (x - y)(x - z)^{n-1}
	\]
	is a majority term (moreover, every variety of rings which admits a majority term is contained in a variety of rings satisfying $x^n = x$ for some  $n \geqslant 2$, see \cite{michler-willie}). By Pixley's theorem mentioned above, the variety of implicative semi-lattices (also known as Heyting semi-lattices) \cite{bsl} possesses a majority term, since it has both distributive and permutable congruences. 
	
	In this paper, we introduce the notion of a \textit{majority category} --- the categorical counterpart of a variety of algebras admitting a majority term (this notion first appeared under the name of a ``Pixley category'' in a talk given by Z.~Janelidze \cite{talk}). These categories provide a link between the notion of a Mal'tsev category \cite{mal'tsev-categories} and the notion of an arithmetical category \cite{P96, DB-protoarithmetical}, and could bear as strong a relation to the category of lattices, as Mal'tsev categories do to the category of groups. Non-varietal examples of majority categories include the dual of any topos, the category of Von Neumann regular rings and the category of topological lattices.  
	
	We will show, amongst other things, that a Barr exact \cite{barr} category is arithmetical if and only if is both Mal'tsev and a majority category. This is a categorical analogue of Pixley's theorem for varieties of algebras mentioned above.  We first show that in the left-exact context, every (finitely complete) Mal'tsev majority category is necessarily protoarithmetical in the sense of D.~Bourn \cite{DB-protoarithmetical} (Corollary~\ref{cor-maltsev-majority-protoarithmetical} below). This is because every internal groupoid in a majority category is an equivalence relation (Theorem~\ref{thm-internal-groupoid}), but also follows from the fact that any unital majority category is antilinear in the sense of \cite{DB-antilinear}. Then, in the Barr-exact context, we show that the converse of Corollary~\ref{cor-maltsev-majority-protoarithmetical} holds:  a category is (proto)arithmetical  if and only if it is both Mal'tsev and a majority category (Theorem~\ref*{arithmetical-malcev-majority}). We then consider the question of whether, in general, protoarithmetical categories are the same as Mal'tsev majority categories, and answer this question in the negative. One of the basic observations here is that $\Rel_3^{\op}$, the dual of the category of ternary relations  (sets equipped with a ternary relation), is regular, has all limits and colimits, and is not a majority category (although, interestingly, the category of binary relations $\Rel_2^{\op}$ is). Then, the full-subcategory $\Mal(\Rel_3^{\op})$ of Mal'tsev objects (in the sense of \cite{thomas}) in $\Rel_3^{\op}$, is a Mal'tsev category in which every internal groupoid is an equivalence relation, and is therefore protoarithmetical. However, $\Mal(\Rel_3^{\op})$ will turn out not to be a majority category.
	
	Surprisingly, duals of many categories of geometric structures such as topological spaces, ordered sets, as well as metric spaces (with sub-contractions), tend to be comajority categories. This raises the question of whether there are categories which are simultaneously majority and comajority categories. We show that preorders are the only such categories among categories with finite limits and binary coproducts. This result is similar to the fact that a category $\C$ such that $\C$ and $\C^{\op}$ is distributive (in the sense of \cite{distributive-category}) is a preorder.
	
	\section{Definition and examples} \label{sec-def-and-examples}
	The presence of a majority term in a variety of algebras, is a condition which may be reformulated for an abstract category, using the so-called ``matrix method'' due to Z.~Janelidze (see \cite{closed1}). This method formulates the condition of a variety admitting a term satisfying some ``elementary equations'', in terms of a certain ``closedness property'' of internal relations in the variety, which is a categorical notion. For example: a Mal'tsev term $q(x,y,z)$ is a ternary term satisfying the equations 
	\begin{align*}
	&q(x_1,x_1,x_2) = x_2,  \\
	&q(y_2,y_1,y_1)  = y_2.
	\end{align*}
	These equations canonically determine an extended matrix of terms in the sense of \cite{closed1}:
	\[
	M =\left(\!\!\begin{array}{ccc|c}
	x_1 & x_1 & x_2 & x_2\\
	y_2 & y_1 & y_1 & y_2
	\end{array}\!\!\right)
	\]
	Recall that in a category $\C$, an internal binary relation $R$ between objects $X$ and $Y$ is a triple $(R_0, r_1,r_2)$, where $r_1:R_0 \rightarrow X$ and $r_2:R_0 \rightarrow Y$ are jointly monomorphic morphisms. If $x:S \rightarrow X$ and $y:S \rightarrow Y$ are any morphisms, we say that the pair $(x,y)$ is $R$-related if there exists a morphism $f:S \rightarrow R_0$ such that $r_1 f = x$ and $r_2 f= y$. Then $R$ is said to be (strictly) $M$-closed if for any morphisms $x_1,x_2:S \rightarrow X$ and $y_1,y_2:S \rightarrow Y$, if $(x_1,y_2), (x_1,y_1)$ and $(x_2,y_1)$ are $R$-related, then $(x_2,y_2)$ is $R$-related. If $R$ satisfies this property, then $R$ is said to be \textit{difunctional}. A finitely complete category $\C$ where every internal relation is difunctional is a Mal'tsev category in the sense of \cite{difunctional} (see also \cite{mal'tsev-categories} for the original notion). In this paper we follow \cite{approximate} and call a category $\C$ (not necessarily finitely complete) Mal'tsev, when every internal relation in $\C$ is difunctional.
	
	The general theory of closedness properties of internal relations (the ``matrix method''), provides a unified way in which to  establish general theorems of categories defined by such a matrix condition. In this setting, there is a general Bourn-localization theorem (see \cite{closed2}), which generalizes, for example, the fact that a finitely complete category $\C$ is Mal'tsev if and only if the fibres $\Pt_I(\C)$ of the fibration of points, are unital (see \cite{DB-unital} and Example~\ref{example-category-of-points} below). Examples of categories defined by such a matrix condition includes subtractive \cite{subtractive}, unital, strongly unital \cite{DB-unital}, and of course, Mal'tsev categories (see \cite{closed1, closed2, closed4}). The definition of a majority category adds to this list, by applying the matrix method to the majority term equations ($*$) given on the first page. 
	
	A ternary relation between objects $A,B$ and $C$ is a quadruple $R = (R_0,r_1,r_2,r_3)$ where $r_1: R_0 \rightarrow A$, $r_2: R_0 \rightarrow B$ and $r_3: R_0 \rightarrow C$ are jointly monomorphic morphisms. If $a:S \rightarrow A $, $b:S \rightarrow B$ and $c:S \rightarrow C$ are any morphisms in $\C$, then we shall say that the triple  $(a,b,c)$ is $R$-related if there exists a morphism $f:S \rightarrow R_0$ such that $r_1 f = a$, $r_2 f = b$ and $r_3 f = c$. 
	\begin{definition}\label{def-majority}
		A  category $\C$ is a majority category when every internal relation in $\C$ is strictly $M$-closed (in the sense of \cite{closed1} with:
		\[M = \left(\!\!\begin{array}{ccc|c}
		a_1 & a_1 & a_2 & a_1\\
		b_1 & b_2 & b_1 & b_1 \\
		c_2 & c_1 & c_1 & c_1
		\end{array}\!\!\right).
		\]
	\end{definition}
	That is to say $\C$ satisfies the following condition:
	\begin{itemize}
		\item[(\textbf{M})]  For any ternary relation $R= (R_0,r_1,r_2,r_3)$ in $\C$ and arbitrary morphisms $a_1,a_2: S \rightarrow A$, $b_1,b_2: S \rightarrow B$ and $c_1,c_2: S \rightarrow C$ in $\C$, if $(a_1,b_1,c_2), (a_1,b_2,c_1)$ and $(a_2,b_1,c_1)$ are $R$-related, then $(a_1,b_1,c_1)$ is $R$-related. 
	\end{itemize}
	In a category with binary products, the condition (\textbf{M}) simply states that a necessary and sufficient condition for a morphism $(a_1,b_1,c_1):S \rightarrow A \times B \times C$ factors through $R$ is that there exist $a_2:S \rightarrow A$,$b_2:S \rightarrow B$ and $c_2:S \rightarrow C$, such that $(a_1,b_1,c_2)$, $(a_1,b_2,c_1)$ and  $(a_2,b_1,c_1)$ factors through $R$.
	\subsection{Examples of Majority Categories}\label{sec-Examples}
	We shall say that a category $\C$ has \textit{image factorizations} if every  morphism $f:X \rightarrow Y$ in $\C$ factors as $f = me$ where $m$ is a monomorphism and $e$ a strong epimorphism. Then the factorization $f = me$ is called an image factorization of $f$.
	\[
	\xymatrix{
		X \ar@/_1pc/[rr]_f \ar[r]^{e}& \bullet  \ar[r]^m & Y	
	}
	\]	 
	We say that a category $\C$ has \textit{co-image factorizations} if $\C^{\op}$ has image factorizations. 
	
	The following theorem characterizes majority categories which have image factorizations. It is a straightforward adaptation  of a result in  \cite{thomas} (Proposition 2.3), which will be used to determine some of the examples of majority categories that follow. 
	\begin{theorem} \label{thm-maj-characterization-images}
		Let $\C$ be a category with image factorizations, binary coproducts and binary products. Then the following are equivalent:
		
		\begin{itemize}
			\item [(1)] $\C$ is a majority category;
			\item [(2)] For any object $S$ in $\C$, there exists a morphism $f:S \rightarrow R$ making the diagram
			\[
			\xymatrix{
				3S \ar[rd]^e\ar[dd]_{M =\begin{pmatrix}
					\iota_1& \iota_1& \iota_2 \\
					\iota_1& \iota_2& \iota_1 \\
					\iota_2& \iota_1& \iota_1
					\end{pmatrix}} & & \\
				& R \ar@{>->}[dl]^r\\
				(2S)^3 & & S \ar[ll]^{(\iota_1, \iota_1, \iota_1)} \ar@{..>}[lu]_f
			}
			\]
			commute, where $ M = r e $ is an image factorization.
			
		\end{itemize}
	\end{theorem}
	\begin{proof}
		Composing $e$ with each of the canonical inclusions $S \rightarrow 3S$, and applying the fact that $\C$ is a majority category,  we have (1) implies (2). We show (2) implies (1): let $\C$ be a category with image factorizations and binary products and binary coproducts. Let $A,B,C$ be any objects in $\C$ and  $r':R' \rightarrowtail A\times B\times C$ any ternary relation. Suppose that $a_1,a_2 \in  \hom(S,A)$, $b_1,b_2 \in  \hom(S,B), c_1,c_2 \in \hom(S,C)$ and $f_1,f_2,f_3 \in \hom(S,R')$ are such that
		\[
		\xymatrix{
			& R' \ar[d]^{r'} & & R' \ar[d]^{r'}  & & R' \ar[d]^{r'} \\
			S \ar[ru]^{f_3} \ar[r]_-{(a_1,b_1,c_2)} &  A \times B \times C & S \ar[ru]^{f_2} \ar[r]_-{(a_1,b_2,c_1)} & A \times B \times C & S \ar[ru]^{f_1} \ar[r]_-{(a_2,b_1,c_1)} & A \times B \times C
		}
		\]
		commute. This implies that the dotted arrow $f$ exists, making the diagram
		\[
		\xymatrix{
			3S \ar@/_2pc/[dd]_M \ar@{..>}[rd]^{f} \ar@/^2pc/[rrdd]^{\begin{pmatrix}
				a_1 & b_1 & c_2  \\
				a_1 & b_2 & c_1  \\
				a_2 & b_1 & c_1
				\end{pmatrix}} \ar[d]_{e} & &  \\
			R \ar@{ >->}[d]_{r}& R' \ar@{  >->}[dr]^{r'} \\
			(2S)^3 \ar[rr]_-{\small\begin{pmatrix}
				a_1 \\
				a_2
				\end{pmatrix} \times \small\begin{pmatrix}
				b_1 \\
				b_2
				\end{pmatrix} \times \small\begin{pmatrix}
				c_1 \\
				c_2
				\end{pmatrix}}  &  & A \times B \times C
		}
		\]
		commute. By assumption, we have that $(\iota_1, \iota_1, \iota_1):S \rightarrow (2S)^3$ factors through $R$ ($\alpha$ in the diagram below), and also by the fact that $M = re$ is an image-factorization, there exists $\beta: R \rightarrow R'$ making the diagram
		\[
		\xymatrix{
			&&3S\ar[rd]^{f}\ar@/^2pc/[rrdd]^{\begin{pmatrix}
				a_1 & b_1 & c_2  \\
				a_1 & b_2 & c_1  \\
				a_2 & b_1 & c_1
				\end{pmatrix}} \ar[d]_{e} & &  \\
			&&
			R\ar@{..>}[r]_{\beta} \ar@{ >->}[d]_{r}& R' \ar@{  >->}[dr]^{r'} \\
			S\ar@{..>}[rru]^{\alpha}\ar[rr]_{(\iota_1,\iota_1,  \iota_1)} &&
			(2S)^3 \ar[rr]_{\small\begin{pmatrix}
				a_1 \\
				a_2
				\end{pmatrix} \times \small\begin{pmatrix}
				b_1 \\
				b_2
				\end{pmatrix} \times \small\begin{pmatrix}
				c_1 \\
				c_2
				\end{pmatrix}}  &  & A \times B \times C
		}
		\]
		commute. Then $r'(\beta \alpha)$ is a factorization of $(a_1,b_1,c_1)$ through $R'$.
	\end{proof}
	\noindent
	By the dual of the theorem above, to verify that $\C^{\op}$ is a majority category, where $\C^{\op}$ has image factorizations, binary products and binary coproducts, it suffices to show the existence of the morphism $f:R \rightarrow S$ making the diagram
	\[
	\xymatrix{
		S^3   & &  \\
		& R \ar[lu]_{e} \ar@{..>}[dr]^f \\
		3S^2 \ar[uu]^{\begin{pmatrix}
			\pi_1& \pi_1& \pi_2 \\
			\pi_1& \pi_2& \pi_1 \\
			\pi_2& \pi_1& \pi_1
			\end{pmatrix}}  \ar@{->>}[ru]^{r} \ar[rr]_{\begin{pmatrix}
			\pi_1 \\
			\pi_1 \\
			\pi_1
			\end{pmatrix}} & & S
	}
	\]
	in $\C$ commute, 
	where $re$ is a co-image factorization of the vertical morphism. This will be done to establish the three examples that follow.
	
	\begin{example} \label{example-Top}
		$\Top^{\op}$ has image factorizations (since it is a  regular regular category). In the above diagram, we may take $R$ to be the set-theoretic image of the vertical morphism equipped with the subspace topology on $S^3$. Then $R$ is given by
		\[
		R = \{(x,x,y) \mid x,y \in S\} \cup \{(x,y,x) \mid x,y \in S\} \cup \{(y,x,x) \mid x,y \in S\}.
		\]
		The morphism $e$ is the canonical inclusion of $R$ into $S^3$, and $r$ is the projection onto the image of the vertical morphism. If $f$ exists, it must satisfy
		\[
		f(x,x,y) = f(x,y,x) = f(y,x,x) = x,
		\]
		since the bottom triangle commutes. Therefore, $\Top^{\op}$ is a comajority category if and only if  $f$ above is continuous for any space $S$: given an open set $U \subseteq S$,
		\[
		f^{-1}(U) = R \cap \big((U \times U \times S) \cup (U \times S \times U) \cup(S \times U \times U)\big).
		\]
	\end{example}
	
	\begin{example} \label{example-rel-2}
		The category $\Rel_2$ has as its objects pairs $(X, \rho_X)$, where $X$ is a set and $\rho_X$ is a binary relation on $X$. A morphism $f:(X, \rho_X)  \rightarrow (Y, \rho_Y)$ is simply a function $f:X \rightarrow Y$ for which:
		\[
		x \rho_X y \implies f(x) \rho_Y f(y)
		\]
		--- such functions are called \textit{monotone}. Similarly as in Example~\ref{example-Top},  for any object $S$ in $\Rel_2$, the map $f:R \rightarrow S$ defined by
		\[
		f(x,x,y) = f(x,y,x) = f(y,x,x) = x
		\]
		is monotone, where
		\[
		R = \{(x,x,y) \mid x,y \in S\} \cup \{(x,y,x) \mid x,y \in S\} \cup \{(y,x,x) \mid x,y \in S\},
		\]
		equipped with the restriction of $\rho_{S^3}$. This is easily verified.
	\end{example}
	\begin{remark} \label{rem-ternary-relations}
		Although the category $\Rel_2$ of sets equipped with binary relations is a comajority category, the category of sets equipped with ternary relations $\Rel_3$ (where morphisms preserve the ternary relation) is not a comajority category (see Section~\ref{sec-relation-to-protoarithmetical}).
	\end{remark}
	\begin{example}
		As shown in $\cite{thomas}$, the category of (extended) metric spaces $\Met_{\infty}$ is coregular, and has products and coproducts. The co-image factorization of a morphism is given by the projection onto the closure of the set-theoretic image $f(X)$ followed by the inclusion into Y:
		\[
		X \rightarrow \overline{f(X)} \rightarrow Y.
		\]
		Given an (extended) metric space $S$, the image of the vertical morphism in the diagram above is given by
		\[
		R = \{(x,x,y) \mid x,y \in S\} \cup \{(x,y,x) \mid x,y \in S\} \cup \{(y,x,x) \mid x,y \in S\},
		\]
		which may be checked to be a closed subset of $S^3$. Therefore, it again suffices to show that $f:R \rightarrow S$ defined by
		\[
		f(x,x,y) = f(x,y,x) = f(y,x,x) = x,
		\]
		is a subcontraction --- which is easily verified. Thus $\Met_{\infty}$ is a comajority category. Using similar arguments as in Theorem~4.3 in \cite{thomas}, it will follow that $\Met$ is too a comajority category.
	\end{example}
	\begin{example}
		By Corollary~\ref{arithmetical-malcev-majority}, any arithmetical category in the sense of \cite{P96,DB-protoarithmetical} is a majority category, so that in particular the dual of every topos, the category of Von Neumann regular rings, Heyting algebras, ect, are all majority categories.
	\end{example}
	\begin{example}
		If $\C$ is a category with products, then the category of $\text{M}(\C)$ of internal majority algebras is a majority category. Therefore, the category $\text{M}(\textbf{Pos})$ of internal majority algebras in the category of partially ordered sets is a majority category. In fact, it can be shown that this is an example of a majority category that is not regular. 
	\end{example}
	\begin{remark}
		A partial order is said to by dually-directed if every pair of elements have a lower and upper bound. Interestingly, if we consider the category $\mathbf{DPos}$ of dually-directed partial orders, then $\text{M}(\textbf{DPos}) \simeq \Lat$ --- the variety of lattices. 
	\end{remark}
	\begin{example}
		It is easy to see that any preorder is a majority category. 
	\end{example}
	Examples of categories which are neither majority nor comajority categories include $\mathbf{Cat}$ the category of all small categories, and also the category of monoids $\mathbf{Mon}$ or groups $\mathbf{Grp}$.
	The next theorem is a special case of Theorem~3.2 in \cite{closed2}.
	\begin{proposition} \label{thm-majority-reflect}
		Suppose that $\C$ and $\mathbb{D}$ are finitely complete categories, and let $F:\mathbb{D} \rightarrow \C$ be a pullback-preserving functor which reflects isomorphisms. Then, if $\C$ is a majority category, then so is $\mathbb{D}$.
	\end{proposition}
	As a consequence of this theorem, we have that if $\C$ is a finitely complete majority category and $X$ any object in $\C$, then both comma categories $(X \downarrow \C)$ and $(\C \downarrow X)$ are majority categories. This is because the forgetful functors $(X \downarrow \C) \rightarrow \C$ and $(\C\downarrow X) \rightarrow \C$ preserve pullbacks and reflect isomorphisms. Also, if $\C$  and $\mathbb{D}$ are categories, with $\C$ a finitely complete majority category, then $\C^{\mathbb{D}}$ is a majority category. 
	\begin{example} \label{example-category-of-points}
		Given a category $\C$ and an object $I$ in $\C$, the category of points $\Pt_I(\C)$ over $I$ has as its objects pairs $(p,s)$ where  $p:X \rightarrow I$ is a split epimorphisms with a chosen splitting $s$. A morphism $f:(p,s) \rightarrow (q,t)$ in $\Pt_I(\C)$ is a morphism in $\C$ such that $qf = p$ and $fs = t$  (see \cite{DB-unital} and \cite{borceux-bourn} for details). If $\C$ has finite limits then so does $\Pt_I(\C)$, and the domain functor $\Pt_I(\C) \rightarrow \C$ which takes $(p,s)$ to the domain of $p$, satisfies the conditions of Proposition~\ref{thm-majority-reflect}. Thus if $\C$ is a finitely complete majority category, then $\Pt_I(\C)$ is a pointed finitely complete majority category for any object $I$ in $\C$.
	\end{example}
	
	\section{Relation to arithmetical, protoarithmetical and antilinear categories} \label{sec-relation-to-protoarithmetical}
	The notion of an \textit{arithmetical} category was first introduced by M.~C.~ Pedicchio in \cite{P96}, as a Barr-exact Mal'tsev category with coequalizers, which is congruence distributive. It was proved there that in an arithmetical category, every internal groupoid is an equivalence relation, moreover this property characterizes arithmetical categories among Barr-exact Mal'tsev categories with coequalizers. Examples of such categories are the dual of any topos, as well as the categories of Boolean algebras and Heyting algebras. In \cite{DB-protoarithmetical}, the author introduces the notion of a \textit{protoarithmetical} category, which is the same as a finitely complete Mal'tsev category in which every internal groupoid is an equivalence relation. In the Barr-exact context, protoarithmetical categories are characterized as congruence distributive Mal'tsev categories. Thus in \cite{DB-protoarithmetical}, an arithmetical category is a Barr-exact Mal'tsev category which is congruence distributive (dropping coequalizers from the original definition), which is what we will mean by arithmetical category. This section shows that in the Barr-exact context, arithmetical categories are precisely Mal'tsev majority categories. And that in general, a protoarithmetical category need not be a majority category.
	\begin{definition} \label{def-protoarithmetical}
		A protoarithmetical category is a finitely complete Mal'tsev category in which every internal groupoid is an equivalence relation.
	\end{definition}
	\begin{remark} \label{rem-definition-protoarithmetical}
		The orginal definition of a protoarithmetical category, which is equivalent to Definition~\ref{def-protoarithmetical}, is that of a finitely complete category $\C$ where the category of points $\Pt_I(\C)$ above any object $I$ is unital \cite{DB-unital}, and such that every internal group in $\Pt_I(\C)$ is trivial. 
	\end{remark}
	One of the main results of \cite{DB-protoarithmetical} is the following Theorem.
	\begin{theorem}[\cite{DB-protoarithmetical}] \label{thm-protoarithmetical-implies-arithmetical}
		A Barr exact category $\C$ is protoarithmetical if and only if it is Mal'tsev and congruence distributive (i.e.~it is arithmetical)
	\end{theorem}
	\begin{theorem} \label{thm-internal-groupoid}
		Every internal groupoid in a majority category is an equivalence relation.
	\end{theorem}
	\begin{proof}
		Suppose that the diagram
		\[
		\xymatrix{
			G_2 \ar[r]^-m & G_1 \ar@(ul,ur)^{\sigma} \ar@/^1pc/[r]^{d_1} \ar@/_1pc/[r]_{d_0} & G_0 \ar[l]|-s	
		}
		\]
		is an internal groupoid in a majority category $\C$,  then we show that $d_1$ and $d_2$ are jointly monomorphic. Let $p_1:G_2 \rightarrow G_1$ and $p_2:G_2 \rightarrow G_1$ be the canonical pullback projections.  Then $R = (G_2,p_1,p_2,m)$ is a ternary relation in $\C$, since $p_1$ and $p_2$ are jointly-monomorphic. Suppose that $f,g: S \rightarrow G_1$ are morphisms with $d_1 f = d_1 g$ and $d_0 f = d_0 g$, then $(f,\sigma f, s d_1 f)$ and $(g,\sigma g, s d_1 g)$ and $(f, \sigma g, m(f,\sigma g))$ are all $R$-related, and hence so is $(f,\sigma g, s d_1 g)$ so that  $m(f,\sigma g) = s d_1 g$, which implies $f = g$.
	\end{proof}
	\begin{corollary} \label{cor-maltsev-majority-protoarithmetical}
		Every finitely complete Mal'tsev majority category is protoarithmetical.
	\end{corollary}
	
	\begin{definition}[\cite{DB-antilinear}] \label{def-cooperator-central}
		Let $\C$ be a pointed category with binary products, and let $f: X \rightarrow Z$ and $g:Y \rightarrow Z$ be morphisms in $\C$. A morphism $\phi:X \times Y \rightarrow Z$ making the diagram
		\[
		\xymatrix{
			X \ar[rd]_-f\ar[r]^-{\iota_X} &  X \times Y \ar[d]|-\phi &  Y\ar[l]_-{\iota_Y} \ar[dl]^-{g}\\
			& Z 
		}
		\] 
		commute, is called a cooperator for $f$ and $g$. If $g = 1_Z$ in the diagram above, then $f$ is said to be central when such a $\phi$ exists.
	\end{definition}
    \begin{definition}[\cite{DB-antilinear}] \label{def-antilinear}
    A unital category $\C$ is said to be antilinear if the only central morphisms are the null morphisms.
    \end{definition}
	\begin{proposition} \label{thm-cooperators}
		Let $\C$ be a pointed finitely complete majority category, and let $f: X \rightarrow Z$ and $g:Y \rightarrow Z$ be morphisms in $\C$. If $f$ and $g$ admit a cooperator, then the square
		\[
		\xymatrix{
			\ker(f) \times \ker(g) \ar[d]_{p_1} \ar[r]^-{p_2} & Y \ar[d]^g \\
			X \ar[r]_f & Z
		}
		\]
		is a pullback. Where $p_1$ and $p_2$ are the canonical product projections composed with the canonical inclusions.
	\end{proposition}
    In particular, this gives that every unital majority category is antilinear in the sense Definition~\ref{def-antilinear}, as the next corollary shows.
    \begin{corollary} \label{cor-central}
If $\C$ is a pointed finitely complete majority category, then $f:X \rightarrow Y$ is central if and only if $f = 0$.
	\end{corollary}
    \begin{proof}
	By Definition~\ref{def-cooperator-central}, $f$ being central, it cooperates with the identity on $Y$, so that by Proposition~\ref{thm-cooperators} the pullback of $1_Y$ along $f$ is given by $\ker(f) \times \ker(1_Y) \simeq \ker(f)$. This implies that the identity on $1_X$ is the kernel of $f$, so that $f = 0$.  
	\end{proof}
	\begin{proof}[Proof of Proposition~\ref{thm-cooperators}]
		Suppose that $\phi$ is a cooperator between $f$ and $g$, then it suffices to show that for any commutative square
		\[
		\xymatrix{
			A \ar[d]_{\alpha} \ar[r]^{\beta} & Y \ar[d]^g \\
			X \ar[r]_f & Z
		}
		\]
		we have $g \beta = 0 = f \alpha$. Consider the ternary relation $r: R \rightarrow X \times Y \times Z$ defined by the equalizer:
		\[
		\xymatrix{
			R \ar[r]^-r & X \times Y \times Z \ar@<-.5ex>[rr]_-{\pi_3} \ar@<.5ex>[rr]^-{\phi(\pi_1, \pi_2)} & & Z
		}
		\]
		Then since we have $\phi(\alpha, 0) = f \alpha$ and $\phi(0, \beta) = g \beta$, by the universal property of the equalizer it follows that 
		$(\alpha, 0, f\alpha): A \rightarrow X \times Y \times Z$ and $(0, \beta, g\beta): A \rightarrow X \times Y \times Z$ and $(0, 0, 0): A \rightarrow X \times Y \times Z$ are all $R$-related. Since $f \alpha = g \beta$, we have that $(0,0,f \alpha)$ is $R$-related, which implies that $f\alpha = 0 = g \beta$.
	\end{proof}
	\begin{remark}
 Corollary~\ref{cor-central} gives another way to see that every finitely complete Mal'tsev majority category is protoarithmetical. If $\C$ is a Mal'tsev majority category, then the category $\Pt_I(\C)$ of points above any object $I$ in $\C$ is unital (see \cite{DB-unital}), and a pointed majority category (see Example~\ref{example-category-of-points}). Thus $\Pt_I(\C)$ is antilinear, and therefore internal monoids in $\Pt_I(\C)$ are trivial. By Remark~\ref{rem-definition-protoarithmetical}, $\C$ is protoarithmetical.
	\end{remark}
	\subsection{Relations in regular categories}
	Recall that if $\C$ is a regular category, then we can define compositions of relations as follows. Let $(r_1,r_2): R \rightarrowtail X \times Y$ and $(s_1,s_2): Y \times Z$ be relations in $\C$, and suppose that $(P, p_1, p_2)$ is the pullback of $s_1$ along $r_2$:
	
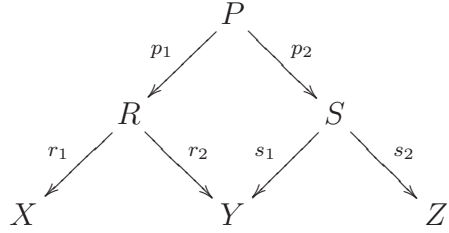
\[
	\xymatrix{
		& &P \ar[dl]_{p_1} \ar[dr]^{p_2} & & \\
		& R \ar[dr]^{r_2} \ar[dl]_{r_1} & & S \ar[dl]_{s_1}\ar[dr]^{s_2} & \\
		X & & Y & & Z
	}
	\]
	The composite $r \circ s: R \circ S \rightarrowtail X \times Z$ is a relation obtained by taking the regular image of $(r_1p_1, r_2p_2): P \rightarrow X \times Z$ as in the diagram:
	\[
	\xymatrix{
		P \ar@/_1pc/[rr]_-{(r_1p_1, r_2p_2)} \ar@{->>}[r]&  R \circ S \ar@{ >->}[r]^-{r \circ s} & X \times Z
	}
	\]
	We have the following lemma for this relation composition. 
	\begin{lemma} \label{lem-relation-composition}
		If $(x,z): S \rightarrow X \times Z$ is any morphism which factors through $R\circ S$, then there exists a regular epimorphism $\alpha: Q \rightarrow S$ and a $y:Q \rightarrow Y$ such that $(x\alpha, y):Q \rightarrow X \times Y$ factors through $R$ and $(y, z\alpha)~:~ Q ~\rightarrow ~Y \times Z$ factors through $S$. 
	\end{lemma}
	\begin{theorem} \label{regular-majority}
		If $\C$ is a regular Mal'tsev category such that the lattice of equivalence relations on each object is a distributive lattice, then $\C$ is a majority category.
	\end{theorem}
	\begin{proof}
		Let $\C$ be a regular Mal'tsev category, such that the lattice of equivalence relations on any object in $\C$ is distributive. Recall that in a regular Mal'tsev category, the join of two congruences is given by their composition. Let 
		
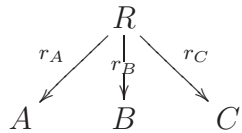
\[
		\xymatrix{
			& R \ar[d]|-{r_B} \ar[rd]^{r_C} \ar[ld]_{r_A}& \\
			A & B & C
		}
		\]
		be any internal ternary relation in $\C$, and let $a_1,a_2: S \rightarrow A, b_1,b_2: S \rightarrow B, c_1,c_2:S \rightarrow C$ and $a,b,c:S \rightarrow R$ be any morphisms in $\C$ such that the diagrams:
		
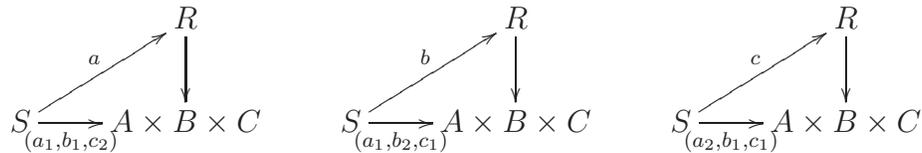
\[
		\xymatrix{
			& R \ar[d] & & R \ar[d]  & & R \ar[d]  \\
			S \ar[ru]^{a} \ar[r]_-{(a_1,b_1,c_2)} & A \times B \times C & S \ar[ru]^{b} \ar[r]_-{(a_1,b_2,c_1)} & A \times B \times C & S \ar[ru]^{c} \ar[r]_-{(a_2,b_1,c_1)} & A \times B \times C
		}
		\]
		commute. Consider the kernel congruences $K_A, K_B, K_C$ on $R$ formed from taking the kernel pairs of $r_A, r_B, r_C$ respectively. Then $(a,c):S \rightarrow R \times R$ factors through $K_B \cap (K_A \circ K_C)$ which implies that $(a,c)$  factors through $(K_B \cap K_A) \circ (K_B \cap K_C)$. By Lemma  \ref{lem-relation-composition},  there exists a regular epimorphism $\alpha:Q \rightarrow S$ and a morphism $b: Q \rightarrow R$ such that $(a \alpha, b)$ factors through $(K_B \cap K_A)$ and $(b, c \alpha )$ factors through $(K_B \cap K_C)$. This implies that $a_1 \alpha = r_A b $ and $b_1 \alpha = r_B b$ and $c_1 \alpha =  r_C b$, and therefore we have the commutative diagram
		\[
		\xymatrix{
			Q \ar[r]^-{b} \ar[d]_{\alpha} & R \ar[d]^-{(r_A,r_B,r_C)} \\
			S \ar[r]_-{(a_1,b_1,c_1)} \ar@{..>}[ru]^-{f} & A \times B \times C
		}
		\]
		where $f$ exists, since $\alpha$ is a regular epimorphism.
	\end{proof}
	\begin{corollary} \label{arithmetical-malcev-majority}
		For a Barr exact category $\C$ the following are equivalent:
		\begin{itemize}
			\item[(1)] $\C$ is arithmetical (i.e.~Mal'tsev and congruence distributive);
			\item[(2)] $\C$ is Mal'tsev and a majority category.
		\end{itemize}
	\end{corollary}
	\begin{proof}
		(1) $\Rightarrow$ (2) is immediate by Theorem~\ref{regular-majority}. For (2) $\Rightarrow$ (1) suppose that $\C$ is a Mal'tsev majority category, then by Theorem~\ref{cor-maltsev-majority-protoarithmetical} we have that $\C$ is protoarithmetical, and thus $\C$ is arithmetical by Theorem~\ref{thm-protoarithmetical-implies-arithmetical}.
	\end{proof}
	\begin{remark}
		This result was actually first announced in \cite{talk}, as a Barr exact analogue of Pixley's result for varieties \cite{pixley}.
	\end{remark}
	\noindent
	The above corollary motivates the question of whether protoarithmetical categories are, in general, the same as Mal'tsev majority categories. Or if there are naturally weaker conditions (than Barr exactness) under which ``Malt'sev + majority = arithmetical''. The rest of this section is dedicated to answering this question in the negative. We will construct a regular protoarithmetical category, with all limits and colimits, which is not a majority category. 
	
	Consider the category of ternary relations $\Rel_3$ mentioned in Example~\ref{rem-ternary-relations}. This category has as its objects pairs  $X = (U_X,R_X)$ where $U_X$ is a set and $R_X$ is a ternary relation on $U_X$.  A morphism $f:X \rightarrow Y$ in $\Rel_3$ is a function $f:U_X \rightarrow U_Y$ for which $(x,y,z) \in R_X \implies (f(x), f(y), f(z)) \in R_Y$.  The limit/colimit of a diagram $D$ in $\Rel_3$ has as its underlying set $U_L$ the set-theoretic limit/colimit  of the underlying diagram is $\Set$,  equipped with the largest/smallest relation making the canonical projections/inclusions homomorphisms. A morphism $m:A \rightarrow X$ in $\Rel_3$ is a regular monomorphism if and only if $m$ is relation-reflecting, which is to say $m$ satisfies
	\[
	(m(x),m(y),m(z)) \in R_X \implies (x,y,z) \in R_A
	\]
	for any $x,y,z \in U_A$.  It may be checked that regular monomorphisms are stable under pushout, which is an easy consequence of the fact that pushouts along monomorphisms in $\Set$ are pullbacks. Therefore we have the following lemma: 
	\begin{lemma} \label{lem-rel-is-regular} 
		The category $\Rel_3^{\op}$ is a complete and cocomplete regular category.
	\end{lemma}
	\begin{remark} \label{rem-epi-regular-mono-factorizations}
		For any morphism $f:X \rightarrow Y$ in $\Rel_3$ denote $f(X)$ for the subrelation of $Y$ restricted to the set-theoretic image of $f$. Then the coimage factorization of $f$ is given by $f = me$ where $e:X \rightarrow f(X)$ is the canonical projection, and $m:f(X) \rightarrow Y$ is the canonical inclusion. 
	\end{remark}
	\begin{definition}[\cite{thomas}]
		Let $S$ be an object in a category $\C$, then $S$ is a Mal'tsev object in $\C$ if for any binary relation $r:R \rightarrow X \times Y$, the induced relation on sets
		\[
		\hom(S,R) \rightarrowtail \hom(S, X) \times \hom(S,Y)
		\]
		is difunctional. 
	\end{definition}
	\begin{remark}[\cite{thomas}]
		A topological space $S$ is a Mal'tsev object in $\Top^{\op}$ if and only if the map $f: R \rightarrow S$ defined by $f(x,x,y) = y = f(y,x,x)$ is continuous, where $R$ is the subspace generated by 
		\[
		\{(x,x,y),(y,x,x) \mid x,y \in S\}.
		\]
		This happens if and only if the space $S$ is an $R_1$-space, which is to say $S$ satisfies the separation axiom: for any $x,y \in S$ if there exists an open $U$ such that $x \in U$ and $y \notin U$, then there exists $V$ and $W$ open, such that $x \in V$ and $y \in W$, and $V \cap W = \varnothing$. Furthermore, a metric space $S$ is a Mal'tsev object in $\Met^{\op}$ if and only if it is an ultra-metric space.
	\end{remark}
	In what follows we will be concerned with Mal'tsev objects in $\Rel_3^{\op}$.
	\begin{lemma} \label{lem-maltsev-relations}
		Let $S$ be any object in $\Rel_3$, and let $M = (U_M,R_M)$ be the subrelation of $S \times S \times S$ where 
		\[
		U_M = \{(x,x,y) \mid x,y \in S\} \cup  \{(y,x,x) \mid x,y \in S\}
		\]
		and $R_M$ is the restriction of $R_{S^3}$ to $U_M$. Then $S$ is a Mal'tsev object in $\Rel_3^{\op}$ if and only if the map $f:U_M \rightarrow U_S$ defined by 
		\[
		f(x,x,y) = y = f(y,x,x)
		\]
		preserves the relation structure (is a morphism in $\Rel_3$).
	\end{lemma}
	\begin{proof}[Sketch]
		By Proposition~2.3 in \cite{thomas}, an object $S$ in $\Rel_3^{\op}$ is a Mal'tsev object if and only if there exists $f: M \rightarrow S$ making the diagram
		\[
		\xymatrix{
			S^3 & \\
			& M \ar[lu]_m \ar@{..>}[dr]^f& \\
			2S^2\ar[uu]^{\begin{pmatrix}
				\pi_2 & \pi_2 & \pi_1 \\
				\pi_1 & \pi_2 & \pi_2
				\end{pmatrix}} \ar[ur]^e \ar[rr]_{\begin{pmatrix}
				\pi_1 \\
				\pi_1
				\end{pmatrix}} & & S
		}
		\]
		in $\Rel_3$ commute, where $me$ is an image-factorization of the vertical morphism. Now by Remark~\ref{rem-epi-regular-mono-factorizations}, $M$ can be taken the be set-theoretic image of the vertical morphism, together with the restriction of $R_{S^3}$. Then
		\[
		U_M = \{(x,x,y) \mid x,y \in S\} \cup  \{(y,x,x) \mid x,y \in S\},
		\]
		and if $f$ exists it must be defined by 
		\[
		f(x,x,y) = y = f(y,x,x).
		\]
	\end{proof}
	The full subcategory of Mal'tsev objects in a category $\C$ is denoted by $\Mal(\C)$, and has the following properties (see \cite{thomas}):
	\begin{enumerate}[label=(\roman*)]
		\item $\Mal(\C)$ is closed under colimits and regular quotients in $\C$. So that in particular if $\C$ is cocomplete, then so is $\Mal(\C)$.
		\item If $\C$ is a regular well-powered category admitting coproducts, then $\Mal(\C)$ is a coreflective subcategory of $\C$.
		\item If $\C$ is a regular category with binary coproducts, such that every morphism in $\Mal(\C)$ which is a regular epimorphism in $\C$ is a regular epimorphism in $\Mal(\C)$, then $\Mal(\C)$ is the largest full subcategory of $\C$ which is Mal'tsev, and, closed under binary coproducts and regular quotients in $\C$.
	\end{enumerate}
	By Lemma~\ref{lem-rel-is-regular} and (ii) above, $\Mal(\Rel_3^{\op})$ is a coreflective subcategory of $\Rel_3^{\op}$. Explicitly, this coreflection $r: \Rel_3^{\op} \rightarrow \Mal(\Rel_3^{\op})$ acts on objects as follows: if $X$ is an object of $\Rel_3^{\op}$, then define $U_{r(X)} = U_X$, and define $R_{r(X)}$ as the smallest ternary relation $R$ on $U_X$ such that $R_X \subseteq R$ and $(U_X, R)$ is a Mal'tsev object in $\Rel_3^{\op}$. Then it can be checked that $r(X)$ is indeed a Mal'tsev object in $\Rel_3^{\op}$. If $f:X \rightarrow Y$ is a morphism in $\Rel_3^{\op}$ then we define $r(f) = f$. To summarize, we have the following lemma:
	\begin{lemma} \label{lem-coreflective}
		The functor $r: \Rel_3^{\op} \rightarrow \Mal(\Rel_3^{\op})$ is right adjoint to the inclusion functor $\iota: \Mal(\Rel_3^{\op}) \rightarrow \Rel_3^{\op}$, and for any object $X$ in $\Rel_3^{\op}$ we have $U_{r(X)} = U_X$.
	\end{lemma}
	The above lemma implies that $\Mal(\Rel_3^{\op})$ has limits, and that the limit of any diagram  $D$ in $\Mal(\Rel_3^{\op})$ has the same underlying set as the corresponding limit of $D$ in $\Rel_3^{\op}$ --- which itself has the same underlying set as the corresponding limit in $\Set^{\op}$. This is to say that the forgetful functor $U:\Mal(\Rel_3^{\op}) \rightarrow \Set^{\op}$ preserves limits. Since every discrete relation ($X$ is discrete if $R_X = U_X \times U_X \times U_X$) is an object of $\Mal(\Rel_3^{\op})$, it will follow that a morphism in $\Mal(\Rel_3^{\op})$ is a monomorphism if and only if it is a monomorphism in $\Rel_3^{\op}$. This implies that the  forgetful functor $\Mal(\Rel_3^{\op}) \rightarrow \Set^{\op}$ reflects monos. Thus we have the following lemma:
	\begin{lemma} \label{lem-forgetful-preserves-limits-reflects-monos}
		The forgetful functor $U:\Mal(\Rel_3^{\op}) \rightarrow \Set^{\op}$ preserves limits and reflects monos.
	\end{lemma}
	\begin{proposition}
		The category $\Mal(\Rel_3^{\op})$ is a complete and cocomplete regular protoarithmetical category. 
	\end{proposition}
	\begin{proof}
		Again, since $\Mal(\Rel_3^{\op})$ contains all discrete relations, it will follow that every morphism in $\Mal(\Rel_3^{\op})$ which is a regular epimorphism in $\Rel_3^{\op}$ is also a regular epimorphism in $\Mal(\Rel_3^{\op})$. Moreover, since $\Mal(\Rel_3^{\op})$ is coreflective, it follows that a morphism in $\Mal(\Rel_3^{\op})$ is a regular epi if and only if it is a regular epi in $\Rel_3^{\op}$. Therefore, since $\Rel_3^{\op}$ is regular, so is $\Mal(\Rel_3^{\op})$. Also, by (iii) above, it follows that $\Mal(\Rel_3^{\op})$ is a Mal'tsev category, and by (i) it is cocomplete. By Lemma~\ref{lem-coreflective},  $\Mal(\Rel_3^{\op})$ inherits its completeness from $\Rel_3^{\op}$.   Next, we show that  any internal groupoid in $\Mal(\Rel_3^{\op})$ is an equivalence relation. Suppose that $G$ is an internal groupoid in $\Mal(\Rel_3^{\op})$, where $G_1$ is the object of arrows and $d_0,d_1:G_1 \rightarrow G_0$ the domain and codomain morphisms respectively. By Lemma~\ref{lem-forgetful-preserves-limits-reflects-monos}, the forgetful functor $U:\Mal(\Rel_3^{\op}) \rightarrow \Set^{\op}$ preserves limits, so that $UG$ is an internal groupoid in $\Set^{\op}$ --- which is a majority category. Thus, $U(d_0,d_1)$ is a monomorphism by Theorem~\ref{thm-internal-groupoid}, and thus $(d_0,d_1)$ is a monomorphism since $U$ reflects monos. 
	\end{proof}
	\begin{definition}
An object $S$ in a category $\C$ is a majority object if for every internal ternary relation $R = (R_0,r_1,r_2,r_3)$ the induced relation on sets
\[
\hom(S,R_0) \rightarrow \hom(S,A) \times \hom(S,B) \times \hom(S,C)
\]
is strictly $M$-closed with $M$ the matrix in Definition~\ref{def-majority}. The full subcategory of majority objects in $\C$ is denoted by $\Maj(\C)$.
\end{definition}
The proposition below is an analogue of Theorem~\ref{thm-maj-characterization-images} for majority objects. The proof is similar to the proof of Theorem~\ref{thm-maj-characterization-images}, and is the analogue of Proposition~2.3 in \cite{thomas} for majority objects. 
\begin{proposition}
Let $\C$ be a category with binary products, binary coproducts and image factorizations. Then for an object $S$ the following are equivalent: 
	\begin{itemize}
	\item [(1)] $S$ is a majority object.
	\item [(2)]  There exists a morphism $f:S \rightarrow R$ making the diagram
	\[
	\xymatrix{
		3S \ar[rd]^e\ar[dd]_{M =\begin{pmatrix}
			\iota_1& \iota_1& \iota_2 \\
			\iota_1& \iota_2& \iota_1 \\
			\iota_2& \iota_1& \iota_1
			\end{pmatrix}} & & \\
		& R \ar@{>->}[dl]^r\\
		(2S)^3 & & S \ar[ll]^{(\iota_1, \iota_1, \iota_1)} \ar@{..>}[lu]_f
	}
	\]
	commute, where $ M = r e $ is an image factorization.
\end{itemize}
\end{proposition}
As an easy application of the above proposition to $\Rel_3^{\op}$, we have the following lemma for majority objects,  which corresponds to Lemma~\ref{lem-maltsev-relations} for Mal'tsev objects.
\begin{lemma} \label{majority-in-rel}
A ternary relation $S$ is a majority object in $\Rel_3^{\op}$ if and only if the map $f:U_N \rightarrow U_S$ defined by $f(x,x,y) = f(x,y,x) = f(y,x,x) = x$ is a morphism in $\Rel_3$ where
\[
U_N = \{(x,x,y) \mid x,y \in S\} \cup \{(x,y,x) \mid x,y \in S\} \cup \{(y,x,x) \mid x,y \in S\}
\]
and $R_N$ is the restriction of $R_{S^3}$ to $U_N$.
\end{lemma}
The full subcategory $\Maj(\C)$ behaves analogously to $\Mal(\C)$, and in particular, we have the proposition below. The proof is left out, as it is a straightforward adaptation of the proof of Corollary~2.4 in \cite{thomas}. 
\begin{proposition}
Let $\C$ be a regular category admitting binary coproducts. If $\D$ is a full subcategory of $\C$ which is a majority category, and closed under binary coproducts and regular quotients in $\C$, then $\D \subseteq \Maj(\C)$.
\end{proposition}
	\begin{proposition}
		$\Mal(\Rel_3^{\op})$ is not a majority category.
	\end{proposition}
	\begin{proof}[Sketch]
		Since $\Mal(\Rel_3^{\op})$ is closed under binary products and regular quotients in $\C$, if $\Mal(\Rel_3^{\op})$ were a majority category, then we would have $\Mal(\Rel_3^{\op}) \subseteq \Maj(\Rel_3^{\op})$ by the proposition above. Thus to show that $\Mal(\Rel_3^{\op})$ is not a majority category, it suffices to produce a Mal'tsev object $S$ which is not a majority object. Consider the ternary relation $S$ where $U_S = \{0,1\}$ and 
		\[
		R_S = \{(1,1,0), (0,1,1), (0,0,0)\}.
		\]
		Then it is routine to verify that $S$ satisfies the conditions of Lemma~\ref{lem-maltsev-relations}, and is thus an object of $\Mal(\Rel_3^{\op})$. If the $f$ in the statement of Lemma~\ref{majority-in-rel} above were a morphism in $\Rel_3$, then we would have
		\[
		((1,0,0) , (1,1,0), (0,1,0)) \in R_{S^3} \implies (f(1,0,0), f(1,1,0), f(0,1,0)) = (0,1,0) \in R_S
		\]
		so that $\Mal(\Rel_3^{\op})$ is not a majority category.
	\end{proof}
	\section{Comajority excludes majority} \label{sec-comajority-excludes-majority}
	In Section~\ref{sec-Examples} we saw that many categories of a geometric nature, such as topological spaces, metric spaces, any topos, ect., form comajority categories. This raises the question of whether there are categories which are simultaneously majority and comajority. This section proves that the only finitely complete categories $\C$ with binary coproducts, such that $\C$ and $\C^{\op}$ are majority categories are the preorders having finite meets and joins. This result is similar to the result that if a category $\C$ is such that both $\C$ and $\C^{\op}$ are distributive categories, then $\C$ is a preorder. 
	
	In what follows, by a \textit{majority} algebra we mean a set $X$ equipped with a majority operation $p_X:X^3 \rightarrow X$. A homomorphism of majority algebras $f:(X,p_X) \rightarrow (Y,p_Y)$ is a function $f:X \rightarrow Y$ satisfying $p_Y(f(x),f(y),f(z)) = f(p_X(x,y,z))$. A majority algebra is said to be \textit{commutative} if the majority operation is a homomorphism.
	\begin{lemma}\label{lemA}
		Let $\C$ be a finitely complete majority category and $A$ any object in $\C$, then the morphisms
		\[
		\xymatrix{
			A^3 \ar[rr]^{(\pi_1,\pi_1, \pi_3)} &   & A^3 &  & A^3 \ar[ll]_{(\pi_1,\pi_2, \pi_2)}\\
			& &  A^3 \ar[u]|-{(\pi_3,\pi_2, \pi_3)}
		}
		\]
		are jointly strongly epimorphic.
	\end{lemma}
	\begin{proof}
		Suppose  $r$ is a monomorphism, such that each of the morphisms above factor through $R$:
		\[
		\xymatrix{
			& R \ar[d]^r & & R \ar[d]^r  & & R \ar[d]^r     \\
			A^3 \ar[ru]^{m_1} \ar[r]_-{(\pi_1,  \pi_1, \pi_3)} & A^3 & A^3 \ar[ru]^{m_2} \ar[r]_-{(\pi_1, \pi_2, \pi_2)} & A^3 & A^3 \ar[ru]^{m_3} \ar[r]_-{(\pi_3, \pi_2, \pi_3)} & A^3
		}
		\]
		\noindent
		Then there exists $m: A^3 \to R$ making the diagram
		
		\centerline{
			\xymatrix{
				& R \ar[d] &    \\
				A^3 \ar@{..>}[ru]^{m} \ar[r]_-{(\pi_1, \pi_2, \pi_3)} & A^3
			}
		}
		\noindent
		commute, so that $r$ is a split epimorphism, and hence an isomorphism. 
	\end{proof}
	
	\begin{lemma} \label{lem-commutative-maj-operation}
		Let $\C$ be a finitely complete majority category with binary coproducts. If $\C$ and $\C^{\op}$ are majority categories, then every $\hom$-set can be equipped with a commutative majority operation.
	\end{lemma}
	\begin{proof}
		Let $A$ be any object of $\C$, then by Lemma~\ref{lemA} the morphism
		\[
		\xymatrix{
			3A^3  \ar[rrr]^-{M_A = \begin{pmatrix}
				\pi_1 & \pi_1 & \pi_3 \\
				\pi_1 & \pi_2 & \pi_2 \\
				\pi_3 & \pi_2 & \pi_3
				\end{pmatrix}}   &  &  &      A^3
		}
		\]
		is an epimorphism. In particular, $A^3$ together with $M_A$ is a ternary corelation on $A^3$ (a ternary relation in $\C^\mathsf{op}$).  Composing $M_A$ with each of the projections $\pi_i:A^3 \rightarrow A$, we have the following commutative diagrams:
		
		\centerline{
			\xymatrix{ 
				3A^3 \ar[rd]_{\colvec{3}{\pi_1}{\pi_1}{\pi_3}}\ar@{->>}[r]^{M_A} & A^3 \ar[d]^{\pi_1} &  3 A^3 \ar[rd]_{\colvec{3}{\pi_1}{\pi_2}{\pi_2}} \ar@{->>}[r]^{M_A} & A^3 \ar[d]^{\pi_2} \ar[d] & 3 A^3 \ar[rd]_{\colvec{3}{\pi_3}{\pi_2}{\pi_3}} \ar@{->>}[r]^{M_A} & A^3 \ar[d]^{\pi_3} \\
				& A & & A & & A
			}
		}
		\noindent
		Since $\C$ is a comajority category, there exists a morphism $p_A:A^3 \rightarrow A$ making the diagram 
		\centerline{
			\xymatrix{
				3A^3 \ar@{->>}[rr]^{M_A} \ar[rrd]_{\small\begin{pmatrix}
						\pi_1 \\
						\pi_2 \\
						\pi_3
				\end{pmatrix}} & &     A^3 \ar@{..>}[d]^{p_A} \\
				& & A
			}
		}
		\noindent
		commute. Thus we have constructed an internal majority operation $p_A$ on $A$, for each object $A$ in $\C$.
		Next, to see that every morphism in $\C$ is a homomorphism with respect to the internal majority operation constructed above, let $f:A \rightarrow B$ be any morphism in $\C$, then the commutativity of the diagram below follows from the commutativity of the top and outer rectangles, and the fact that $M_A$ is an epimorphism.	
		\[
		\xymatrix{
			3A^3 \ar@{->>}[d]^{M_A}\ar[rr]^{3f^3} & & 3B^3 \ar@{->>}[d]^{M_B}\\
			A^3 \ar[d]^{p_A} \ar[rr]^{f^3} & & B^3 \ar[d]^{p_B} \\
			A \ar[rr]^f & & B \\
		}
		\]
		The commutativity of the bottom square is precisely the statement that $f$ is a homomorphism with respect to the internal majority operations $p_A$ and $p_B$. Therefore, for any objects $S$ and $A$ the composite
		\[
		\hom(S,A)^3 \simeq \hom(S,A^3) \xrightarrow{\hom(S,p_A)} \hom(S,A)
		\]
		is a commutative majority operation.
	\end{proof}
	\begin{lemma} \label{lem-comutative-maj-algebra}
		Let $(X,p_X)$ be a commutative majority algebra, then $X$ has at most one element.
	\end{lemma}
	\begin{proof}
		Let $x,y \in X$ be any two elements, then
		\begin{align*}
		x &= p_X(x,x,y) \\
		&= p_X(p_X(x,x,y),p_X(x,y,x), p_X(y,y,y)) \\
		&= p_X(p_{X^3}((x,x,y),(x,y,y), (y,x,y))) \\
		&= p_X(p_X(x,x,y),p_X(x,y,y),p_X(y,x,y)) \\
		&= p_X(x,y,y) = y
		\end{align*}
	\end{proof}
	\noindent
	As an immediate corollary of Lemma~\ref{lem-commutative-maj-operation} and Lemma~\ref{lem-comutative-maj-algebra}, we have:
	\begin{theorem} \label{thm-majority-excludes-comajority}
		If $\C$ has finite limits and binary coproducts, and $\C$ and $\C^{\op}$ are majority categories, then $\C$ is a preorder. 
	\end{theorem}
	\begin{remark}
		It is possible to prove the theorem above under different limit and colimit assumptions, but it is impossible without at least some limit and colimit assumptions. This is because the the category consisting of just two parallel arrows is both majority and comajority. 
	\end{remark}
	\begin{remark}
		The proof above depends on the fact that the morphisms in the statement of Lemma \ref{lemA} are epimorphic. In a unital category, they are jointly strongly epimorphic. Therefore, by the proof above we may also conclude that a unital category $\C$ with binary coproducts such that $\C$ is comajority is equivalent to the terminal category $\mathbf{1}$.
	\end{remark}
	
	\noindent 
	In \cite{closed2}, the author asks: for a general term matrix $M$, how the following conditions on a category $\C$ are related to each other:
	\begin{itemize}
		\item[(a)] $\C$ is enriched in the variety of commutative $M$-algebras.
		\item[(b)] $\C$ and $\C^{\op}$ are $M$-closed. 
	\end{itemize}
	\noindent
	For the matrices corresponding to unital, subtractive and Mal'tsev categories
	\[\left(\!\!\begin{array}{cc|c}
	x &  0 & x\\
	0 &  x & x
	\end{array}\!\!\right), \qquad
	\left(\!\!\begin{array}{cc|c}
	x &  x & 0\\
	x &  0 & x
	\end{array}\!\!\right), \qquad
	\left(\!\!\begin{array}{ccc|c}
	x &  x & y & y\\
	u &  v & v & u
	\end{array}\!\!\right),
	\]
	(a) and (b) are equivalent under suitable conditions on the base category $\C$ \cite{closed4, closed1}. For instance, if $\C$ is a pointed category with binary products and coproducts, epi-mono factorizations of its morphisms, such that (b) holds for 
	\[
	M = \left(\!\!\begin{array}{cc|c}
	x &  x & 0\\
	x &  0 & x
	\end{array}\!\!\right),
	\]
	then $\C$ is enriched in the category of abelian groups (which are the same as commutative subtraction algebras). For the matrix corresponding to majority categories, we also have (a) equivalent to (b), and the results of this section show this equivalence.
	\subsection*{Acknowledgements}
	I would like to thank my supervisor Professor Z.~Janelidze for encouraging me to introduce and study majority categories. I would also like to thank Professor M.~Gran for suggesting to me to consider internal structures (in particular, internal groupoids) in majority categories, which lead me to Theorem~\ref{thm-internal-groupoid}. Finally, I would like to thank an anonymous referee for many helpful comments, and in particular the suggested improvement of the original proof of Theorem~\ref{thm-internal-groupoid}.
	
\end{document}